\documentclass[ejs]{imsart}
\usepackage{hyperref}

\input{macrogilles_main.tex}
\usepackage{amsmath,amsfonts,amsthm,amssymb}

\usepackage{mathtools}

\newcommand{\mk}{m}

\begin{document}
\begin{frontmatter}

  \title{On the Post Selection Inference constant under Restricted Isometry Properties}
  \runtitle{Post selection inference under RIP}
\date{\today}

\begin{aug}
\author{\fnms{François} \snm{Bachoc}
\ead[label=e1]{francois.bachoc@math.univ-toulouse.fr}}
\address{
Institut de Mathématiques de Toulouse; \\
UMR 5219; Université de Toulouse; CNRS\\
UPS, F-31062 Toulouse Cedex 9, France\\
\printead{e1}\\
}


\author{\fnms{Gilles} \snm{Blanchard}
\ead[label=e2]{gilles.blanchard@math.uni-potsdam.de}}
\address{
Universit\"at Potsdam, Institut f\"ur Mathematik\\
Karl-Liebknecht-Stra{\ss}e 24-25 14476 Potsdam, Germany\\
\printead{e2}\\
}


\author{\fnms{Pierre} \snm{Neuvial}
\ead[label=e3]{pierre.neuvial@math.univ-toulouse.fr}}
\address{
Institut de Mathématiques de Toulouse; \\
UMR 5219; Université de Toulouse; CNRS\\
UPS, F-31062 Toulouse Cedex 9, France\\
\printead{e3}\\
}

\runauthor{F. Bachoc, G. Blanchard, P. Neuvial}

\end{aug}

\maketitle

  \begin{abstract}
    Uniformly valid confidence intervals post model selection in
    regression can be constructed based on Post-Selection Inference
    (PoSI) constants.
    PoSI constants are minimal for orthogonal design
    matrices, and can be upper bounded in function of the sparsity of
    the set of models under consideration, for generic design matrices.

    In order to improve on these generic sparse upper bounds, we consider design matrices satisfying a Restricted Isometry Property (RIP) condition. We provide a new upper bound on the PoSI constant in this setting. This upper bound is an explicit function of the RIP constant of the design matrix, thereby giving an interpolation between the orthogonal setting and the generic sparse setting.  We show that this upper bound is asymptotically optimal in many settings by constructing a matching lower bound.
  \end{abstract}



\end{frontmatter}

\section{Introduction}

Fitting a statistical model to data is often preceded by a model selection step. The construction of valid statistical procedures in such post model selection situations is quite challenging (cf. \cite{leeb05model, leeb2006, leeb08model}, \cite{kabaila06large} and \cite{Poe09a}, and the references given in that literature), and has recently attracted a considerable amount of attention. Among
various recent references in this context, we can mention those addressing sparse high dimensional settings with a focus on lasso-type model selection procedures \cite{Bel11a, Bel14a, van14a,Zha14a}, those aiming for conditional coverage properties for polyhedral-type model selection procedures \cite{Fit15a,lee15exact,LeeTay14,Tib15a,tibshirani2014exact} and those achieving valid post selection inference universally over the model selection procedure \cite{bachoc14valid,bachoc2016uniformly,Berk13}.

In this paper, we shall focus on the latter type of approach and
adopt the setting introduced in \cite{Berk13}. In that work, a
linear Gaussian regression model is considered, based on an $n \times p$ design matrix $X$.
A model $M \subset \{1,...,p \}$ is defined as a subset of indices of the $p$ covariates. For a family $\cM \subset \set{ M | M \subset \set{1,\ldots,p}}$ of admissible models, it is shown in \cite{Berk13} that a universal coverage property is achievable (see Section \ref{section:setting:notation}) by using a family of confidence intervals whose sizes are proportional to a constant $K(X,\cM) >0$. This constant $K(X,\cM)$ is called a PoSI (Post-Selection Inference) constant in \cite{Berk13}. This setting was later extended to prediction problems in \cite{bachoc14valid} and to misspecified non-linear settings in \cite{bachoc2016uniformly}.

The focus of this paper is on the order of magnitude of the PoSI constant $K(X,\cM)$ for large $p$. We shall consider $n \geq p$ for simplicity of exposition in the rest of this section
(and asymptotics $n,p\rightarrow \infty$). It is shown in \cite{Berk13} that $K(X,\cM)= \Omega(\sqrt{\log(p)})$; this rate is reached in particular when $X$ has orthogonal columns. On the other hand, in full generality
$K(X,\cM)=O(\sqrt{p})$ for all $X$. It can also be shown, as discussed in an intermediary version of \cite{zhang17spherical}, that when $\cM$ is composed of $s$-sparse submodels, the sharper
upper bound  $K(X,\cM) = O(\sqrt{s \log(p/s)})$ holds.
Hence, intuitively, design matrices that are close to orthogonal and consideration of sparse models yield smaller PoSI constants.

In this paper, we obtain additional quantitative insights for this intuition, by considering design matrices $X$ satisfying restricted isometry property (RIP) conditions. RIP conditions have become central in high dimensional statistics and compressed sensing \cite{buhlmann2011statistics,candes2005decoding,foucart2013mathematical}.
In the $s$-sparse setting and for design matrices $X$ that satisfy a RIP property of order $s$ with RIP constant $\delta \to 0$, we show that $K(X,\cM)=
O(\sqrt{\log(p)} + \delta \sqrt{s \log(p/s)})$. This corresponds to the intuition that for such matrices,
any subset of $s$ columns of $X$ is ``approximately orthogonal''.
Thus, under the RIP condition we improve the upper bound of \cite{zhang17spherical} for the $s$-sparse case, by up to a factor $\delta \to 0$. We show that our upper bound is complementary to the bounds recently proposed in \cite{kuchibhotla2018model}. In addition, we obtain lower bounds on $K(X,\cM)$ for a class of design matrices that extends the equi-correlated design matrix in \cite{Berk13}. From these lower bounds,
we show that the new upper bound we provide is optimal, in a large range of situations.

While the main interest of our results is theoretical, our suggested upper bound can be practically useful in cases where it is computable whereas the PoSI constant $K(X,\mathcal{M})$ is not. The only challenge for computing our upper bound is to find a value $\delta$ for which the design matrix $X$ satisfies a RIP property. While this is currently challenging in general for large $p$, we discuss, in this paper, specific cases where it is feasible.

The rest of the paper is organized as follows. In Section \ref{section:setting:notation} we introduce in more details the setting and the PoSI constant $K(X,\cM)$. In Section \ref{sec:upper-bound-rip} we introduce the RIP condition, provide the upper bound on $K(X,\cM)$ and discuss its theoretical comparison with \cite{kuchibhotla2018model} and its applicability. In Section \ref{sec:lower-bound} we provide the lower bound and the optimality result for the upper bound. All the proofs are given in the appendix.

\section{Settings and notation} \label{section:setting:notation}

\subsection{PoSI confidence intervals}

We consider and review briefly the framework introduced by \cite{Berk13} for which
the so-called PoSI constant plays a central role.
The goal is to construct post-model selection confidence intervals that are agnostic with respect
to the model selection mehod used. The authors of \cite{Berk13} assume a Gaussian vector of observations
\begin{equation} \label{eq:gaussian:linear:model}
Y = \mu + \epsilon,
\end{equation}
where the $n \times 1$ mean vector $\mu$ is fixed and unknown, and $\epsilon$ follows the $\mathcal{N}(0, \sigma^2 I_n)$ distribution where $\sigma^2 >0$ is unknown. Consider an $n \times p$ fixed design matrix $X$, whose columns correspond to explanatory variables for $\mu$. It is not necessarily assumed that $\mu$ belongs to the image of $X$ or that $n \geq p$.

A model $M$ corresponds to a subset of selected variables in $\set{1,\ldots,p}$. A set of models of interest $\cM \subset \cM_{all} = \set{ M | M\subset \set{1,\ldots,p}}$ is supposed to be given. Following \cite{Berk13}, for any $M \in \cM$, the projection based vector of regression coefficients $\beta_M$ is a target of inference, with
 \begin{equation}
\beta_M := \argmin_{\beta \in \mathbb{R}^{|M|}} \norm{\mu - X_M\beta}^2 =
(X_M^t X_M)^{-1} X_M^t \mu,
\label{eq:beta-M}
\end{equation}
where $X_M$ is the submatrix of $X$ formed of the columns of $X$ with indices in $M$, and
where we assume that for each $M \in \cM$, $X_M$ has full rank and $M$ is non-empty. We refer to \cite{Berk13} for an interpretation of the vector $\beta_M$ and a justification for considering it as a target of inference. In \cite{Berk13}, a family of confidence intervals $( \mathrm{CI}_{i,M} ; i \in M \in \cM)$ for $\beta_M$ is introduced, containing the targets $(\beta_M)_{M \in \cM}$ 
simultaneously with probability at least $1-\alpha$. The confidence intervals take the form
\begin{equation} \label{eq:CI}
\mathrm{CI}_{i,M}
:=
(\hat{\beta}_M)_{i.M}
\pm
\hat{\sigma}
\| v_{M,i} \|
K(X,\cM,\alpha,r);
\end{equation}
the different quantities involved, which we now define, are standard ingredients for univariate
confidence intervals for regression coefficients in the Gaussian model, except for the last factor (the ``PoSI constant'') which will account for multiplicity of covariates and models, and their
simultaneous coverage.
The confidence interval is centered at
$\hat{\beta}_M := (X_M^t X_M)^{-1} X_M^t Y$, the ordinary least squares estimator
of $\beta_M$; also, if $M = \{ j_1,\ldots,j_{|M|} \}$ with $j_1 <\ldots <j_{|M|}$,
for $i\in M$ we denote by $i.M$ the number $k \in \mathbb{N}$ for which $j_k = i$, that is, the rank of the $i$-th element in the subset $M$.
The quantity $\hat{\sigma}^2$ is an unbiased estimator of $\sigma^2$, more specifically it
is assumed that it is an observable random variable, such that $\hat{\sigma}^{2}$
is independent of $P_{X}Y$ and is distributed as $\sigma ^{2}/r$
times a chi-square distributed random variable with $r$ degrees of freedom
($P_{X}$ denoting the orthogonal projection onto the column space of $X$). We allow for $r = \infty$ corresponding to $\hat{\sigma} = \sigma$,
i.e., the case of known variance (also called Gaussian limiting case). In \cite{Berk13}, it is assumed that $\hat{\sigma}$ exists and it is shown that this indeed holds in some specific situations. A further analysis of the existence of $\hat{\sigma}$ is provided in \cite{bachoc14valid,bachoc2016uniformly}.

The next quantity to define is
\begin{equation} \label{eq:vMi}
v_{M,i}^t
:=
(e_{i.M}^{|M|})^t G_M^{-1} X_M^t \in \mbr^n,
\end{equation}
where   $e_a^{b}$ is the $a$-th base column vector of $\mathbb{R}^b$; and
$G_M := X_M^t X_M$ is the $|M| \times |M|$ Gram matrix formed from the columns of $X_M$.
Observe that $v_{M,i}$ is nothing more than the row corresponding to covariate $i$ in
the estimation matrix $G_M^{-1} X_M^t$, in other words $(\hat{\beta}_M)_{i.M} = v_{M,i}^tY$.

Finally,  $K(X,\cM,\alpha,r)$ is called a PoSI constant and we turn to its definition.
We shall occasionally write for simplicity $K(X,\cM,\alpha,r) = K( X , \cM )$. Furthermore, if the value of $r$ is not specified in $K( X , \cM )$, it is implicit that $r = \infty$.

\begin{definition} \label{def:Kposi}
  Let $\cM \subset \cM_{all}$ for which each $M\in \cM$ is non-empty, and so that $X_M$ has full rank. Let also
\[
w_{M,i}
=
\begin{cases}
v_{M,i} / \|v_{M,i}\|, & \mbox{if } \|v_{M,i}\| \neq 0; \\
0 \in \mathbb{R}^n  & \mbox{else}.
\end{cases}
\]
Let $\xi$ be a Gaussian vector with zero mean vector and identity covariance matrix on $\mathbb{R}^n$. Let $N$ be a random variable, independent of $\xi$, and so that $r N^2$ follows a chi-square distribution with $r$ degrees of freedom. If $r = \infty$, then we let $N=1$. For $\alpha \in (0,1)$, $K( X , \cM , \alpha , r )$ is defined as the $1 - \alpha$ quantile of
\begin{equation}
\label{eq:gamma_Mr}
\gamma_{\cM,r} := \frac{1}{N} \max_{M \in \cM, i \in M} \left| w_{M,i}^t \xi \right|.
\end{equation}

\end{definition}
We remark that $K( X , \cM , \alpha , r )$ is the same as in \cite{Berk13}. For $j = 1,\ldots,p$, let $X_j$ be the column $j$ of $X$.
We also remark, from \cite{Berk13}, that the vector
$v_{M,i}/\norm{v_{M,i}}^2$ in \eqref{eq:vMi} is the residual of the regression of $X_{i}$ with respect to the variables $\set{j | j \in M\setminus\set{i}}$;
in other words, it is the component of the vector $X_{i}$ orthogonal to $\mathrm{Span}\set{ X_j | j \in M\setminus\set{i}}$.
It is shown in \cite{Berk13} that we have, with probability larger than $1 - \alpha$,
\begin{equation} \label{eq:coverage}
\forall M \in \cM,
~ ~
\forall i \in M,
~ ~
(\beta_M)_{i.M} \in \mathrm{CI}_{i,M}.
\end{equation}
Hence, the PoSI confidence intervals guarantee a simultaneous coverage of all the projection-based regression coefficients, over all models $M$ in the set $\cM$.

For a square symmetric non-negative matrix $A$, we let 
$$\mathrm{corr}(A) = (\mathrm{diag}(A)^\dagger)^{1/2} A (\mathrm{diag}(A)^\dagger)^{1/2},$$
where $\mathrm{diag}(A)$ is obtained by setting all the non-diagonal elements of $A$ to zero and where $B^\dagger$ is the Moore-Penrose pseudo-inverse of $B$. Then we show in the following lemma that $K(X,\cM)$ depends on $X$ only through $\mathrm{corr}( X^t X)$.
\begin{lemma} \label{lem:posi:corrX}
Let $X$ and $Z$ be two $n \times p$ and $m \times p$ matrices satisfying the relation $\mathrm{corr}( X^t X) = \mathrm{corr}( Z^t Z)$. Then $K( X , \cM , \alpha , r ) = K( Z , \cM , \alpha , r )$.
\end{lemma}

\subsection{Order of magnitude of the PoSI constant}

The confidence intervals in \eqref{eq:CI} are similar in form to the standard confidence intervals that one would use for a single fixed model $M$ and a fixed $i \in M$. For a standard interval, $K( X , \cM )$ would be replaced by a standard Gaussian or Student quantile. Of course, the standard intervals do not account for multiplicity and do not have uniform coverage over $i \in M \in \cM$ (see \cite{bachoc14valid,bachoc2016uniformly}). Hence $K( X , \cM )$ is the inflation factor or correction over standard intervals to get uniform coverage; it must go to infinity as $p \to \infty$ \cite{Berk13}. Studying the asymptotic order of magnitude of $K( X , \cM )$ is thus an important problem, as this order of magnitude corresponds to the price one has to pay in order to obtain universally valid post model selection inference.

We now present the existing results on the asymptotic order of magnitude of $K( X , \cM )$.
Let us define
\begin{equation}
\gamma_{\cM, \infty} :=  \max_{M \in \cM, i \in M} \left| w_{M,i}^t \xi \right|\,,
\end{equation}
so that $\gamma_{\cM,r} = \gamma_{\cM, \infty}/N$, where we recall that $r N^2$ follows a chi-square distribution with $r$ degrees of freedom.

We can relate the quantiles of $\gamma_{\cM,r}$ (which coincide with the PoSI constants $K( X , \cM )$) to the expectation $\e{\gamma_{\cM, \infty}}$
 by the following argument based on
Gaussian concentration (see Appendix~\ref{sec:gauss-conc}):
\begin{proposition}
  \label{prop:boundkt}
  Let $T(\mu,r,\alpha)$ denote the $\alpha$-quantile of a noncentral $T$ distribution with $r$
  degrees of freedom and noncentrality parameter $\mu$. Then
  \[
    K( X , \cM , \alpha , r ) \leq T(\e{\gamma_{\cM,\infty}},r,1-\alpha/2).
  \]
  To be more concrete, we observe that we can get a rough estimate of the latter quantile via
  \[
    T(\e{\gamma_{\cM,\infty}},r,1-\alpha/2) \leq     \frac{ \e{\gamma_{\cM,\infty}} + \sqrt{2 \log (4/\alpha)}}{(1- 2\sqrt{2\log (4/\alpha)/r})_+};
  \]
  furthermore, as $r \to +\infty$, this quantile reduces to the  $(1-\alpha/2)$ quantile of a
  Gaussian distribution with mean $\e{\gamma_{\cM,\infty}}$ and unit variance.
  \end{proposition}

  The point of the above estimate is that the dependence in the set of models $\cM$
  is only present through $\e{\gamma_{\cM,\infty}}$. Therefore, we will focus in this paper on the problem of bounding $\e{\gamma_{\cM,\infty}}$, which is nothing more than the Gaussian width \cite[chapter 9]{foucart2013mathematical} of the set $\Gamma_{\cM} = \set{\pm w_{M,i} | M \in \cM, i \in M }$.

When $n \geq p$, it is shown in \cite{Berk13} that $\e{\gamma_{\cM,\infty}}$ is no smaller than $\sqrt{2\log(2p)}$ and asymptotically no larger than $\sqrt{p}$. These two lower and upper bound are reached by respectively orthogonal design matrices and equi-correlated design matrices (see \cite{Berk13}).

We now concentrate on $s$-sparse models. For $s\leq p$, let us define $\cM_s = \set{ M | M\subset \set{1,\ldots,p}, |M| \leq s}$.
In this case, using a direct argument based on cardinality, one gets the following generic upper bound (proved in Appendix~\ref{sec:proofs}).
\begin{lemma}
For any $s,n,p \in \mathbb{N}$, with $s \leq n$, we have
\label{lm:simplecard}
  \begin{equation}
    \label{eq:simplecard}
    \e{\gamma_{\cM_{s},\infty}} \leq \sqrt{2  s \log (6p/s)}\,.
  \end{equation}
\end{lemma}

We remark that an asymptotic version of the bound in Lemma \ref{lm:simplecard} (as $p$ and $s$ go to infinity) appears in an intermediary version of \cite{zhang17spherical}.

\section{Upper bound under RIP conditions}
\label{sec:upper-bound-rip}

\subsection{Main result}

We recall the definition and a property of the RIP constant $\kappa(X,s)$ associated to a design matrix $X$ and a sparsity condition $s$ given in \cite[Chap.6]{foucart2013mathematical}: 
\begin{equation}
  \label{eq:def-rip-constant}
  \kappa(X,s) = \sup_{|M|\leq s} \norm{X_M^t X_M - I_{|M|}}_{op}.
\end{equation}
Letting $ \kappa =  \kappa(X,s)$, we have for any subset $M \subset \set{1,\ldots,p}$ such that $\abs{M} \leq s$:
\begin{equation}
  \label{RIP}  
 \forall \beta \in \mbr^{\abs{M}}, \qquad (1-\kappa)_+ \norm{\beta}^2 \leq \norm{X_M \beta }^2 \leq  (1+\kappa) \norm{\beta}^2\,.
\end{equation}
\begin{remark} The RIP condition may also be stated between norms instead of squared norms in \eqref{RIP}. Following \cite[Chap.6]{foucart2013mathematical} we will consider the formulation in terms of squared norms, which is more convenient here.
\end{remark}

Since the PoSI constant $K(X,\cM)$ only depends on $\mathrm{corr}(X^tX)$ (see Lemma \ref{lem:posi:corrX}), we shall rather consider the RIP constant associated to $\mathrm{corr}(X^tX)$. We let
\begin{equation}
  \label{eq:def-rip-constant:corr}
  \delta(X,s) = \sup_{|M|\leq s} \norm{ \mathrm{corr}(X_M^t X_M) - I_{|M|}}_{op}.
\end{equation}

Any upper bound for $\kappa(X,s)$ yields an upper bound for $\delta(X,s)$ as shown in the following lemma.

\begin{lemma}
\label{lem:bar:delta:delta}
Let $\kappa = \kappa(X,s)$. If $\kappa \in [0,1)$, then
\[
\delta(X,s) \leq \frac {2 \kappa}{1 - \kappa}.
\]
\end{lemma}
The next theorem is the main  result of the paper. It provides a new upper bound on the PoSI constant, under RIP conditions and with sparse submodels.
We remark that in  this theorem, we do not necessarily assume that $n \geq p$.

\begin{theorem}
\label{thm:rip-upper-bound}
Let $X$ be a $n \times p$ matrix with $n,p \in \mathbb{N}$.
Let $\delta = \delta(X,s) $.
We have
\[
  \e{\gamma_{\cM_s,\infty}} \leq
\sqrt{2 \log (2p)}
+ 2 \delta \left( \frac{\sqrt{ 1+\delta}}{1-\delta}  \right)
\sqrt{ 2s \log(6p / s) }.
\]
\end{theorem}

This upper bound is of the form
\[
U_{\textrm{RIP}}(p, s, \delta) =
U_{\textrm{orth}}(p)
+ 2 \delta c(\delta)
U_{\textrm{sparse}}(p, s),
\]
where:
\begin{itemize}
\item $U_{\textrm{orth}}(p)=\sqrt{2 \log (2p)}$ is the upper bound in
  the orthogonal case;
\item $U_{\textrm{sparse}}(p, s)$ is the right-hand side of \eqref{eq:simplecard} corresponding to the cardinality-based upper bound in the sparse case;
\item $c(\delta) = \sqrt{1+\delta}/(1-\delta)$ satisfies: $c(\delta) \geq 0$, $c(\delta) \to 1$ as $\delta \to 0$, and $c$ is increasing.
\end{itemize}
We observe that if $\delta \to 0$, our bound $U_{\textrm{RIP}}$ is  $o(U_{\textrm{sparse}})$. Moreover, when $\delta \sqrt{s} \sqrt{ 1 - \log s / \log p + 1 / \log p } \to 0$, then $U_{\textrm{RIP}}$ is even asymptotically equivalent to  $U_{\textrm{orth}}$. In particular, this is the case if $\delta \sqrt{s} \to 0$.

We now consider the specific case where $X$ is a subgaussian random matrix, that is, $X$ has independent subgaussian entries  \cite[Definition 9.1]{foucart2013mathematical}. We discuss in which situations $\delta = \delta(X,s) \to 0$. The estimate of $\kappa$ in \cite[Theorem 9.2]{foucart2013mathematical} combined with Lemma~\ref{lem:bar:delta:delta} yields
\begin{equation}
  \label{eq:RIP:ensemble:gaussian}
  \delta = O_P\left(\sqrt{s \log (ep/s)/n}\right),
\end{equation}
so that $\delta \to 0$ as soon as $n /(s \log (ep/s)) \to +\infty$.

\subsection{Comparison with upper bounds based on Euclidean norms}
\label{sec:comp-kuchibhotla}
We now compare our upper bound in Theorem~\ref{thm:rip-upper-bound} to upper bounds recently and independently obtained in \cite{kuchibhotla2018model}.
Recall the notation $Y$, $\mu$, $\beta_M$ and $\hat{\beta}_M$ from Section \ref{section:setting:notation} and let $r = \infty$ for simplicity of exposition.
The authors in \cite{kuchibhotla2018model} address the case where $X$ is random (random design) and
consider deviations of $\wh{\beta}_M$ to
$\bar\beta_M=\e{X_M^tX_M}^{-1}\e{X_M^tY}$, the population version of the regression coefficients $\beta_M$, assuming that the rows of $X$ are independent
random vectors in dimension $p$.
They derive uniform bounds over $M \in \cM_s$ for $\norm[1]{\bar\beta_M - \hat{\beta}_M}_2$.
They also consider briefly (Remark 4.3 in \cite{kuchibhotla2018model})
the fixed design case with $\beta_M=(X_M^tX_M)^{-1}X_M^t\mu$ as in the present paper. This target $\beta_M$ can be interpreted as the random design model conditional to $X$. They
assume that the individual coordinates of $X$ and $Y$ have exponential moments
bounded by a constant independently from $n,p$ (thus their setting is more general than
the Gaussian regression setting, but for the purpose of this discussion we assume Gaussian noise).

Let us additionally assume that the RIP property 
$\kappa(X/\sqrt{n},s)\leq \kappa$ is satisfied (on an event of probability tending to 1) and for $\kappa$ restricted to a compact of $[0,1)$ independently of $n,p$; note that we used the rescaling of $X$ by $\sqrt{n}$, which is natural in the random
design case.
Then some simple estimates obtained as a consequence of Theorems\footnote{The technical conditions assumed by \cite{kuchibhotla2018model} imply a slightly weaker version of the RIP property $\kappa(X/\sqrt{n},s)\leq \kappa<1$.} 3.1 and 4.1 in \cite{kuchibhotla2018model}
 lead to
\begin{equation}
  \label{eq:them}
  \sup_{M \in \cM_s}  \norm[1]{ \beta_M - \hat{\beta}_M}_2
   = O_P\left(\sigma 
     \sqrt{\frac{s \log(ep/s)}{n}}\right),
 \end{equation}
 as $p,n \rightarrow \infty$ and assuming $s\log^2{p} = o(n)$.
On our side, under the same assumptions we have that
\begin{align*}
\sup_{M \in \cM_s, i \in M}\left(\left(\frac{X_M^tX_M}{n}\right)^{-1}\right)_{i.M i.M}
\end{align*}
is bounded on an event of probability tending to 1. This leads to $\norm{v_{i.M}} = O_P(1/\sqrt{n})$ uniformly for all $M \in \cM_s, i \in M$.
Hence, from Theorem \ref{thm:rip-upper-bound}, \eqref{eq:CI}, \eqref{eq:coverage},
we obtain
\begin{align}
  \label{eq:us1}
  \sup_{M \in \cM_s} \norm[1]{  \beta_M - \hat{\beta}_M }_\infty = O_P\paren{
  \sigma\paren{\sqrt{\frac{\log (p)}{n}}
+ \delta  \sqrt{ \frac{s \log(ep / s) }{n}}}}.
\end{align}
Thus, if $\delta = \Omega(1)$, since the Euclidean norm upper bounds the supremum
norm, the results of  \cite{kuchibhotla2018model} imply ours (at least in the sense
of these asymptotic considerations). On the other hand,
in the case where $\delta\rightarrow 0$, which is the case we are specifically interested
in, we obtain a sharper bound (in the weaker supremum norm).

In particular, if $X$ is a subgaussian random matrix (as discussed in
the previous section),
due to~\eqref{eq:RIP:ensemble:gaussian} we obtain
\begin{align}
  \label{eq:us}
  \sup_{M \in \cM_s} \norm[1]{  \beta_M - \hat{\beta}_M }_\infty = O_P\left(\sigma\paren{\sqrt{\frac{\log(p)}{n}} + \frac{s \log(ep/s)}{n}} \right).
\end{align}
This improves over the estimate deduced from~\eqref{eq:them} as soon as
$s \log(ep/s) = o(n)$, which corresponds to the case where \eqref{eq:them} tends to $0$.
Conversely, in this situation our bound \eqref{eq:us} yields for the Euclidean norm
(using $\norm{w}_2 \leq \norm{w}_0 \norm{w}_\infty$):
\begin{align}
\label{eq:us-euclidean}
\sup_{M \in \cM_s} \norm[1]{\beta_M - \hat{\beta}_M }_2 = O_P\left(\sigma\paren{\sqrt{ \frac{s\log(p)}{n}} + \frac{s^{3/2} \log(ep/s)}{n}} \right) \,.
\end{align}
Assuming $s = O(p^\lambda)$ for some $\lambda < 1$ for ease of interpretation, we see that \eqref{eq:us-euclidean} is of the same order as  \eqref{eq:them} when  $s^2 \log(p) = O(n)$, and is of a strictly larger order otherwise. In this sense, it seems that \eqref{eq:us1} and \eqref{eq:them} are complementary to each other since we are using a weaker norm, but obtain a sharper bound in the case $\delta\rightarrow 0$.

\subsection{Applicability}
\label{sec:applicability}

While the main interest of our results is theoretical, we now discuss the applicability of our bound.
For any $\delta \geq \delta(X, s)$, Theorem~\ref{thm:rip-upper-bound} combined with Proposition~\ref{prop:boundkt} provides a bound of the form $\ol{U}_{\textrm{RIP}}(p, s, \delta) \geq K(X, \cM_s)$, with
\begin{equation*}
  \ol{U}_{\textrm{RIP}}(p, s, \delta) = T\paren{\sqrt{2 \log (2p)}
+ 2 \delta \left( \frac{\sqrt{ 1+\delta}}{1-\delta}  \right)
\sqrt{ 2s \log(6p / s) },r,1-\alpha/2}.
\end{equation*}
 This bound can be used in practice in situations where $\delta(X, s)$ (or an upper bound of it) can be computed, whereas $K(X, \cM_s)$ cannot because the number of inner products in \eqref{eq:gamma_Mr} is too large.  Indeed, for a given $\delta$, it is immediate to compute $\ol{U}_{\textrm{RIP}}(p, s, \delta)$.

\paragraph{Upper bounding the RIP constant.}
When $n \geq p$, we have $\delta(X, s) \leq  \delta(X, p)$ and $\delta(X, p)$ can be computed in practice for a given $X$. Specifically, $\delta(X, p)$  is the largest eigenvalue of $\mathrm{corr}(X^t X) - I_p$ in absolute value. When $X$ is a subgaussian random matrix,  $\delta(X, p) \sim \sqrt{p/n}$ \cite{bai2010spectral,marcenko1967distribution}. Thus, if $n$ is large enough compared to $p$, the computable upper bound $\ol{U}_{\textrm{RIP}}(p, s, \delta(X, p))$ will improve on the sparsity-based upper bound  $\ol{U}_{\textrm{sparse}}(p, s) = T(( 2s \log(6p / s) )^{1/2},r,1-\alpha/2) \geq K(X, \cM_s)$, see Proposition~\ref{prop:boundkt} and Lemma~\ref{lm:simplecard}.

On the other hand, when $n < p$, it is typically too costly to compute $\delta(X, s)$ (or an upper bound of it) for a large $p$. Nevertheless, if one knows that $X$ is a subgaussian random matrix, they can compute an upper bound $\tilde{\delta}$ satisfying  $\tilde{\delta} \geq \delta(X,s)$ with high probability, as in \cite[Chapter 9]{foucart2013mathematical}.
We remark that using the values of $\tilde{\delta}$ currently available in the literature, one would need $n$ to be very large for $\ol{U}_{\textrm{RIP}}(p, s, \tilde{\delta})$ to improve on $\ol{U}_{\textrm{sparse}}(p, s)$.

\paragraph{Alternative  upper bound on the PoSI constant.}
For any $\delta \geq \delta(X,s)$, we now show how to compute an alternative bound of the form $ \tilde{U}_{\textrm{RIP}}(p, s, \delta) \geq K(X, \cM_s)$.  Our numerical experiments suggest that this alternative bound is generally sharper than $\overline{U}_{\textrm{RIP}}(p, s, \delta)$.
For $q,r, \rho \in \mathbb{N}$ and $\ell \in (0,1)$, let $B_{\ell}(q,r,\rho)$ be defined as the smallest $t >0$ so that
\[
\cH_{q,\rho}(t) := \mathbb{E}_G \left( \min\left(1, \rho \left[  1 - F_{\mathrm{Beta},1/2,(q-1)/2} (t^2/G^2) \right] \right) \right) \leq \ell,
\]
where $G^2/q$ follows a Fisher distribution with $q$ and $r$ degrees of freedom, 
and $F_{\mathrm{Beta},a,b}$ denotes the cumulative distribution function of the Beta$(a,b)$ distribution. In the case $r = +\infty$, $B_\ell$ is also defined and further described in \cite[Section 2.5.2]{bachoc2016uniformly}.

It can be seen from the proof of Theorem~\ref{thm:rip-upper-bound} (see specifically \eqref{eq:for-B-alpha} which also holds without the expectation operators),
and from the arguments in \cite{bachoc14valid}, that we have
\begin{equation*}
  K(X, \cM_s, \alpha) \leq B_{t \alpha}(n \wedge p, r, p) +  2\delta  c(\delta) B_{(1-t)\alpha}(n \wedge p, r , \abs{\cM_s})
\end{equation*}
for any $t \in (0,1)$.
This upper bound can be minimized with respect to $t$, yielding $\tilde{U}_{\textrm{RIP}}(p, s, \delta)$.

The quantity $B_\ell(q,r, \rho)$ can be easily approximated numerically, as it is simply the quantile of the tail distribution $\cH_{q,\rho}$, which only involves standard distributions.  Algorithm E.3 in the supplementary materials of \cite{bachoc14valid} can be used to compute $B_\ell(q,r,\rho)$. An implementation of this algorithm in R~\cite{R} is available in Appendix \ref{apx:code}. Hence, the upper bound $\tilde{U}_{\textrm{RIP}}(p, s, \delta)$ can be computed for large values of $p$ for a given $\delta$.

\section{Lower bound}
\label{sec:lower-bound}

\subsection{Equi-correlated design matrices}
\label{sec:equi-corr-design}

The goal of this section is to find a matching lower bound for Theorem~\ref{thm:rip-upper-bound}.
For this we extend ideas of \cite[Example 6.2]{Berk13} and, following that reference, we
restrict our study to design matrices $X$ for which $n \geq p$. The lower bound is based on the $p \times p$ matrix $Z^{(c,k)}= (e_1^p,e_2^p,\ldots,e_{p-1}^p,x_k(c))$, where \[
x_k(c)=(\underbrace{c,c, \dots c}_{k}, \underbrace{0, 0, \dots 0}_{p-1-k}, \underbrace{\sqrt{1-kc^2}}_1)^t,\]
where we assume $k<p$, and
the constant $c$ satisfies $c^2<1/k$, so that $Z^{(c,k)}$ has full rank. By definition, the correlation between any of the first $k$ columns of $Z^{(c,k)}$ and the last one is $c$, and $Z^{(c,k)}$ restricted to its first $p-1$ columns is the identity matrix $I_{p-1}$. The case where $k=p-1$ is studied in \cite[Example 6.2]{Berk13}: Theorem 6.2 in \cite{Berk13} implies that the PoSI constant $K(X,\cM)$, where $X$ is a $n \times p$ matrix such that $X^t X = (Z^{(c,k)})^t Z^{(c,k)}$, is of the order of $\sqrt{p}$ when $k=p-1$ and $\cM = \cM_{all} $.
The Gram matrix of  $Z^{(c,k)}$  is the $3 \times 3$ block matrix with sizes $(k, p-k-1, 1) \times (k, p-k-1, 1)$ defined by

\begin{equation}
\label{eq:gram-xck}
(Z^{(c,k)})^t Z^{(c,k)}
 =
\begin{bmatrix}
    I_{k} &  [0] & [c] \\
    [0] & I_{p-k-1} & [0]\\
    [c] & [0] & 1 \\
\end{bmatrix},
\end{equation}
where $[a]$ means that all the entries of the corresponding block are identical to $a$.
We begin by studying the RIP coefficient $\delta(X,s)$
for design matrices $X$ yielding the Gram matrix \eqref{eq:gram-xck}. Since this Gram matrix has full
rank $p$, there exists a design matrix satisfying this condition if and only if $n \geq p$.

\begin{lemma}\label{lm:rip-constant}
  Let $X$ be a $n \times p$ matrix for which $X^t X$ is given by \eqref{eq:gram-xck} with $kc^2 < 1$. Then for $s \leq k \leq p-1$,  we have $\kappa(X,s) = \delta(X,s)\leq c\sqrt{s-1}$.
\end{lemma}

\subsection{A matching lower bound}

In the following proposition, we provide a lower bound of $K(X , \cM_s)$ for matrices $X$  yielding the Gram matrix \eqref{eq:gram-xck}.

\begin{proposition}\label{prop:matching-lower-bound}
  For any $s \leq k < p$, $c^2 < 1/k$ and $\alpha \leq \frac{1}{2}$, let $X$ be a $n \times p$ matrix for which $X^t X$ is given by \eqref{eq:gram-xck} with $kc^2 < 1$. We have
  \[
    K(X, \cM_s,\alpha, \infty)
    \geq A\frac{c(s-1)}{\sqrt{ 1 - (s-1)c^2 }}  \sqrt{ \log  \lfloor k/s \rfloor } - \sqrt{2\log 2},
  \]
  where $A>0$ is a universal constant.
\end{proposition}

From the previous lemma, we now show that the upper bound of Theorem \ref{thm:rip-upper-bound} is optimal (up to a multiplicative constant) for a large range of behavior of $s$ and $\delta$ relatively to $p$. As discussed after Theorem \ref{thm:rip-upper-bound}, in the case where
$\delta \sqrt{s} \sqrt{ 1 - \log s / \log p + 1/ \log p } = O(1)$,
the upper bound we obtain is optimal, since it can be written as $O( \sqrt{ \log p } )$. In the next Corollary, we show that the upper bound of  Theorem \ref{thm:rip-upper-bound} is also optimal when $\delta \sqrt{s} \sqrt{ 1 - \log s / \log p + 1/ \log p }$ tends to $+\infty$,
and when $\delta = O(p^{-\lambda})$ for some $\lambda>0$.

\begin{corollary}[Optimality of the RIP-PoSI bound]
\label{cor:optimality:bound}
Let $(s_p, \delta_p)_{p \geq 0}$ be sequences of values such that $s_p < p$, $\delta_p >0$, $\delta_p \to 0$ and satisfying:
$$\lim_{p \to \infty} \delta_p \sqrt{s_p} \sqrt{ 1 - \log s_p / \log p + 1 / \log p} = + \infty.$$
Then Theorem \ref{thm:rip-upper-bound} implies

\begin{equation} \label{eq:2}
  \sup_{\substack{n \in \mathbb{N} \\ s\leq s_p,  X \in \mbr^{n \times p} \\ \mbox{s.t. } \delta(X, s) \leq  \delta_p}}  K(X, \cM_{s_p})
  \leq
  B  \delta_p \sqrt{s_p} \sqrt{\log(6p/s_p)},
\end{equation}
where $B$ is a constant.
Moreover, there exists a sequence of design matrices $X_p$  such that $\delta(X_p, s_p) \leq \delta_p$ and
\begin{equation}
  \label{eq:1}
  K(X_p, \cM_{s_p}) \geq   A \delta_p \sqrt{s_p} \sqrt{
  \log
  \left(
  \min( 1/\delta_p^2 , \lfloor (p-1)/s_p \rfloor)
\right)
  } \,,
\end{equation}
where $A$ is a constant.

  In particular, if  $\delta_p = O(p^{-\lambda})$ for some $\lambda>0$ and if $\lfloor (p-1)/s_p \rfloor \geq 2$, then the above upper and lower bounds have the same rate.
\end{corollary}

Therefore, the upper bound in Theorem~\ref{thm:rip-upper-bound} is optimal in most configurations of $s_p$ and $\delta_p$, except if $\delta_p$ goes to $0$ slower than any inverse power of $p$.

\section{Concluding remarks}
\label{sec:concluding-remarks}

In this paper, we have proposed an upper bound on PoSI constants in $s$-sparse situations where the $n \times p$ design matrix $X$ satisfies a RIP condition. As the value of the RIP constant $\delta$ increases from $0$, this upper bound provides an interpolation between the case of an orthogonal $X$  and an existing upper bound only based on sparsity and cardinality. We have shown that our upper bound is asymptotically optimal for many configurations of $(s, \delta, p)$ by giving a matching lower bound. In the case of random design matrices with independent entries, since $\delta$ decreases with $n$, our upper bound compares increasingly more favorably to the cardinality-based upper bound as $n$ gets larger. It is also complementary to the bounds recently proposed in \cite{kuchibhotla2018model}. The interest and various applications of the RIP property
are well-known in the high-dimensional statistics literature, in particular for statistical risk analysis or support recovery. Our analysis puts into light an additional interest of the RIP property for agnostic post-selection inference (uncertainty quantification).

The PoSI constant corresponds to confidence intervals on $\beta_M$ in~\eqref{eq:beta-M}. In section~\ref{sec:comp-kuchibhotla} we also mention another target of interest in the case of random $X$,
$\bar\beta_M=\e{X_M^tX_M}^{-1}\e{X_M^tY}$. This quantity depends on the distribution of $X$ rather than on its realization, which is a desirable property as discussed in \cite{bachoc14valid,kuchibhotla2018model} where the same target has also been considered. In \cite{bachoc14valid}, it is shown that valid confidence intervals for $\beta_M$ are also asymptotically valid for $\bar\beta_M$, provided that $p$ is fixed. These results require that $\mu$ belongs to the column space of $X$ and hold for models $M$ such that $\mu$ is close to the column space of $X_M$. It would be interesting to study whether assuming RIP conditions on $X$ enables to alleviate these assumptions.

The purpose of post-selection inference based on the PoSI constant $K(X,\cM)$ is to achieve the coverage guarantee \eqref{eq:coverage}. The guarantee \eqref{eq:coverage} implies that, for any model selection procedure $\hat{M}: \mathbb{R}^n \to \cM$, with probability larger than $1 - \alpha$,
for all $i \in \hat{M}$, $(\hat{M})_{i.\hat{M}} \in \mathrm{CI}_{i,\hat{M}}$. Hence, there is in general no need to make assumptions about the model selection procedure when using PoSI constants. On the other hand, the RIP condition that we study here is naturally associated to specific model selection procedures, namely the lasso  or the Dantzig selector \cite{candes2007dantzig,candes2005decoding,van2009conditions,zhao2006model}.
Hence, it is natural to ask whether the results in this paper could help post-selection inference specifically for such procedures. We believe that the answer could be positive in some situations.
Indeed, if the lasso model selector is used in conjunction with a design matrix $X$ satisfying a RIP property, then asymptotic guarantees exist on the sparsity of the selected model \cite{buhlmann2011statistics}. Thus, one could investigate the combination of bounds on the size of selected models (of the form $|\hat{M}| \leq S$ and holding with high probability) with our upper bound, by replacing $s$ by $S$.

 In the case of the lasso model selector, we have referred, in the introduction section, to the post-selection intervals achieving conditional coverage \cite{lee15exact}, specifically for the lasso model selector.
These intervals are simple to compute (when the conditioning is on the signs, see \cite{lee15exact}).
Generally speaking, in comparison with confidence intervals based on PoSI constants, the confidence intervals of \cite{lee15exact} have the benefit of guaranteeing a coverage level conditionally on the selected model.
On the other hand the confidence intervals in \cite{lee15exact} can be large, and can provide small coverage rates when the regularization parameter of the lasso is data-dependent \cite{bachoc14valid}. It would be interesting to study whether these general conclusions would be modified in the special case of design matrices satisfying RIP properties.

Finally, the focus of this paper is on PoSI constants in the context of linear regression. Recently, \cite{bachoc2016uniformly} extended the PoSI approach to more general settings (for instance
generalized linear models), provided a joint asymptotic normality property holds between model dependent targets and estimators. This extension was suggested in the case of asymptotics for fixed dimension and fixed number of models.
In the high-dimensional case, an interesting direction would be to apply the results of \cite{chernozhukov2013gaussian}, that provide Gaussian approximations for maxima of sums of high-dimensional random vectors. This opens the perspective of applying our results to various
high-dimensional post model selection settings, beyond linear regression.

\section*{Acknowledgements}

This work has been supported by ANR-16-CE40-0019 (SansSouci).
The second author acknowledges the support from the german DFG, under the
Research Unit FOR-1735 ``Structural Inference in Statistics - Adaptation and
Efficiency'', and under the Collaborative Research Center SFB-1294 ``Data Assimilation''.

\section*{Appendix}
\appendix

\section{Gaussian concentration}
\label{sec:gauss-conc}

To relate the expectation of a supremum of Gaussian variables to its quantiles, we use the following classical Gaussian concentration inequality \cite{CIS76} (see e.g. \cite{Giraud15}, Section B.2.2. for
a short exposition):
\begin{theorem}[Cirel'son, Ibragimov, Sudakov]
  \label{thm:gausconc}
  Assume that $F:\mbr^d \rightarrow \mbr$ is a 1-Lipschitz function (w.r.t. the Euclidean norm of its input) and $Z$ follows the $\cN(0,\sigma^2 I_d)$ distribution. Then, there exists two one-dimensional standard Gaussian variables $\zeta,\zeta'$
  such that
  \begin{equation}
    \label{eq:gaussconc}
    \e{F(Z)} - \sigma \abs{\zeta'} \leq F(Z) \leq \e{F(Z)} + \sigma \abs{\zeta}.
  \end{equation}
\end{theorem}
It is known that in certain situations one can expect an even tighter concentration, through the phenomenon
known as superconcentration \cite{chatterjee2014superconcentration}. While such situations are likely to be relevant for the
setting considered in this paper, we leave such improvements as an open issue for future work.

We use the previous property in our setting as follows:
\begin{proposition}
\label{prop:expectation-to-quantiles}
  Let $\cC$ be finite a family of unit vectors of $\mbr^n$, $\xi$ a standard Gaussian vector
  in $\mbr^n$ and $N$ an independent nonnegative
  random variable so that $rN^2$ follows a chi-squared distribution
  with $r$ degrees of freedom. Define the random variable
  \[
    \gamma_{\cC,r} := \frac{1}{N} \max_{v \in \cC} \abs{v^t\xi}.
  \]
  Then the $(1-\alpha)$ quantile of $\gamma_{\cC,r}$ is upper bounded by the $(1-\alpha/2)$ quantile
  of a noncentral $T$ distribution with $r$ degrees of freedom and noncentrality parameter
  $\e{\max_{v \in \cC} \abs{v^t\xi}}$.
\end{proposition}
\begin{proof}
  Observe that $\xi \mapsto \max_{v \in \cC} \abs{v^t\xi}$ is 1-Lipschitz since the vectors of $\cC$
  are unit vectors. Therefore we conclude by Theorem~\ref{thm:gausconc} that there exists a standard
  normal variable $\zeta$ (which is independent of $N$ since $N$ is independent of $\xi$) so that
  the following holds:
  \[
    \gamma_{\cC} \leq \frac{1}{N}
    \paren[2]{ \e[2]{\max_{v \in \cC} \abs{v^t\xi}} + \abs{\zeta} }.
  \]
  We can represent the above right-hand side as $\max(T_+,T_-)$ where
  \[
    T_\pm = \frac{1}{N}
    \paren[2]{  \e[2]{\max_{v \in \cC} \abs{v^t\xi}} \pm \zeta },
  \]
  i.e. $T_+,T_-$ are two (dependent) noncentral $t$ distributions with $r$ degrees of freedom and
  noncentrality parameter $\e{\max_{v \in \cC} \abs{v^t\xi}}$. Finally since
  \[
    \prob{\max(T_+,T_-) > t} \leq \prob{T_+ > t} + \prob{T_- > t} = 2 \prob{T_+ > t},
  \]
  we obtain the claim.
\end{proof}

Since a noncentral distribution is (stochastically) increasing in its noncentrality parameter,
any bound obtained for $\e{\max_{v \in \cC} \abs{v^t\xi}}$ will result in a corresponding bound
on the quantiles of the corresponding noncentral $T$ distribution and therefore of those of
$\gamma_\cC$. In the limit $r\rightarrow \infty$, the quantiles of the noncentral
$T$ distribution reduce to those of a shifted Gaussian distribution with unit variance.

Here is a naive bound on (some) quantiles of a noncentral $T$:
\begin{lemma}
  The  $1-\alpha$ quantile of a noncentral $T$ distribution with $r$
  degrees of freedom and noncentrality parameter $\mu \geq 0$ is upper bounded by:
  \[
    (\mu + \sqrt{2 \log (2/\alpha)}/(1- 2\sqrt{2\log (2/\alpha)/r})_+.
    \]
\end{lemma}
\begin{proof}
  Let
  \[
    T = \frac{\mu + \zeta}{\sqrt{V/r}}\,,
  \]
  where $\zeta\sim \cN(0,1)$ and $V\sim \chi^2(r)$.
  We have (as a consequence of e.g. \cite{Birge01}, Lemma 8.1),
  for any $\eta \in (0,1]$:
  \[
    \prob{ \sqrt{V} \leq \sqrt{r} - 2 \sqrt{2\log \eta^{-1}}} \leq \eta,
  \]
  as well as the classical bound
  \[
    \prob{ \zeta \geq \sqrt{2 \log \eta^{-1}}} \leq \eta.
  \]
  It follows that
  \[
    \prob{T \geq (\mu + \sqrt{2 \log \eta^{-1}})/(1- 2\sqrt{2\log (\eta^{-1})/r})_+} \leq 2\eta.
    \]
    The claimed estimate follows.
\end{proof}

\section{Proofs}
\label{sec:proofs}

\begin{proof}[Proof of Lemma \ref{lem:posi:corrX}]
With the notation of Definition \ref{def:Kposi},
$K( X , \cM , \alpha , r )$ is the $1 - \alpha$ quantile of $(1/N)\|z\|_{\infty} $ where $z = ( z_{M,i}, M \in \cM, i \in \cM )$
is a Gaussian vector, independent of $N$, with mean vector zero and covariance matrix $\mathrm{corr} ( \Sigma ) $, where $\Sigma$ is defined by, for $i \in M \in \cM$ and $i' \in M' \in \cM$,
\begin{align*}
\Sigma_{(M,i),(M',i')}
& =
v_{M,i}^t v_{M',i'} \\
& =
(e_{i.M}^{|M|})^t
( X_M^tX_M )^{-1}
X_{M}^t X_{M'}
( X_{M'}^tX_{M'} )^{-1}
e_{i'.M'}^{|M'|}.
\end{align*}
Hence, $\Sigma$ depends on $X$ only through $X^t X$. Also, if $X$ is replaced by $X D$, where $D$ is a diagonal matrix with positive components, $\Sigma$ becomes the matrix $\Lambda$ with
for $i \in M \in \cM$ and $i' \in M' \in \cM$,
\begin{align*}
\Lambda_{(M,i),(M',i')}
& =
(e_{i.M}^{|M|})^t
D_{M,M}^{-1}
( X_M^tX_M )^{-1}
X_{M}^t X_{M'}
( X_{M'}^tX_{M'} )^{-1}
D_{M',M'}^{-1}
e_{i'.M'}^{|M'|} \\
& =
D_{i,i}^{-1}
D_{i',i'}^{-1}
\Sigma_{(M,i),(M',i')}.
\end{align*}
Hence, $\mathrm{corr}(\Sigma) = \mathrm{corr}(\Lambda)$.
This shows that $\Sigma$ depends on $X$ only through $\mathrm{corr}(X^t X)$ (we remark that because $\cup_{ \cM} M = \{1,\ldots,p\}$ and each $X_M^t X_M$ is invertible we have that $\|X_i\| > 0$ for $i=1,\ldots,p$). Hence $K( X , \cM , \alpha , r )$ depends on $X$ only through $\mathrm{corr}(X^t X)$.
\end{proof}

\begin{proof}[Proof of Lemma \ref{lm:simplecard}]\label{proof:simplecard}
  Using a direct cardinality-based bound we have the well-known inequality
  $\e{\gamma_{\cM_s,\infty}} \leq  \sqrt{2 \log (2 \abs{\cM_s})}$, hence
  \begin{equation*}
    \e{\gamma_{\cM_{s},\infty}} \leq \sqrt{2 \log \paren[4]{ 2\sum_{i=1}^s  i \binom{p}{i}}} ,
  \end{equation*}
  moreover
  \[
    \sum_{i=1}^s  i \binom{p}{i} \leq s \sum_{i=0}^s \binom{p}{i} \leq s \paren{\frac{pe}{s}}^s,
  \]
  the last inequality being classical and due to
  \[
    \paren{\frac{s}{p}}^s \sum_{i=0}^s \binom{p}{i} \leq \sum_{i=0}^s \paren{\frac{s}{p}}^i \binom{p}{i}
    \leq \paren{1+\frac{s}{p}}^p \leq e^s.
  \]
  Since $\log s \leq s/e$, and using $e^{1+2/e}\leq 6$, we obtain
  \[
    \log \left(2\sum_{i=1}^s  i \binom{p}{i}\right) \leq \log 2s + s \log \paren{\frac{pe}{s}}
    \leq s \log\paren{\frac{p}{s} e^{1+2/e}} \leq s \log\paren{\frac{6p}{s} },
  \]
  implying \eqref{eq:simplecard}.
  \end{proof}

\begin{proof}[Proof of Lemma \ref{lem:bar:delta:delta}]

  Put $\kappa=\kappa (X,s) <1$. Then, $\|X_i\| \geq (1-\kappa)^{1/2}$ for $i=1,...,p$ so that for $i \in M \in \cM_s$, $\mathrm{corr}(X_M^tX_M) = D_M X_M^t X_M D_M$ where $D_M$ is a $|M| \times |M|$ matrix defined by $[D_M]_{i.M,i.M} = 1/\|X_i\|$. Hence $\norm{D_M}_{op} \leq 1/\sqrt{1-\kappa}$. We have, by applications of the triangle inequality and since $\norm{.}_{op}$ is a matrix norm,
\begin{align}
\MoveEqLeft \norm{ \mathrm{corr}(X_M^tX_M) - I_{|M|} }_{op} \notag \\
= & 
\norm{
(D_M - I_{|M|}) X_M^tX_M D_M
+
 X_M^tX_M ( D_M - I_M)
 +
  X_M^tX_M
 - I_{|M|}
}_{op}
\notag
\\
\leq &
\norm{
D_M - I_{|M|}
}_{op}
\norm{
X_M^tX_M
}_{op}
\norm{
D_M
}_{op}
+
\norm{
D_M - I_{|M|}
}_{op}
\norm{
X_M^tX_M
}_{op}
\notag
\\
& +
\norm{
X_M^tX_M - I_{|M|}
}_{op} \notag
\\
 = &
\norm{
D_M - I_{|M|}
}_{op}
\norm{
X_M^tX_M
}_{op}
\paren{\norm{D_M}_{op} + 1}
+ \norm{
  X_M^tX_M - I_{|M|}
}_{op}. \label{eq:upper-bound-rip}
\end{align} 
From \eqref{eq:def-rip-constant}-\eqref{RIP}, we have for all $M\in \cM_s$: $\norm{X_M^tX_M}_{op} \leq 1+\kappa,$ as well as
\begin{align*}
\norm{
D_M - I_{|M|}
}_{op}
\leq &
 \max_{i=1,...,p}
 \left|
 \frac{1}{\|X_i\|}
 -1
 \right| \\
\leq & \max
       \left(
       1-
\frac{1}{\sqrt{1+ \kappa}},
\frac{1}{\sqrt{1- \kappa}}
-1
 \right)\\
= & \frac{1}{\sqrt{1- \kappa}}
-1.
\end{align*}
Plugging this into \eqref{eq:upper-bound-rip}, we obtain

\begin{align*}
  \delta(X, s)
& \leq
\paren{ \frac{1}{\sqrt{1- \kappa}} - 1}
\paren{1 + \kappa}
\paren{ \frac{1}{\sqrt{1- \kappa}} + 1}
+ \kappa \\
&  = \frac{2 \kappa}{1 - \kappa}.
\end{align*}
\end{proof}

\begin{proof}[Proof of Theorem \ref{thm:rip-upper-bound}]
From Lemma \ref{lem:posi:corrX}, it is sufficient to treat the case where, for any $M$, $G_M = X_M^t X_M$ has ones on the diagonal; in that case $\delta(X, s) = \kappa(X,s)$.
We have
\begin{align*}
v_{M,i}^t
& =
(e_{i.M}^{|M|})^t G_M^{-1} X_M^t \\
& =
(e_{i.M}^{|M|})^t I_{|M|} X_M^t
+
(e_{i.M}^{|M|})^t
\left( G_M^{-1} - I_{|M|} \right)
 X_M^t \\
 & =
 X_{i}^t
 + r_{M,i}^t,
\end{align*}
say. We have
\begin{align*}
r_{M,i}^t r_{M,i}
& =
(e_{i.M}^{|M|})^t
\left( G_M^{-1} - I_{|M|} \right) G_M
\left( G_M^{-1} - I_{|M|} \right)
e_{i.M}^{|M|}
 \\
& \leq
\norm[1]{ e_{i.M}^{|M|} }^2
\norm{ G_M^{-1} - I_{|M|}  }_{op}^2
\norm{ G_M }_{op}.
\end{align*}
From \eqref{RIP}, the eigenvalues of $G_M$ are all between $(1-\delta)$ and $(1 + \delta)$, hence we have
\[
r_{M,i}^t r_{M,i}
\leq
\left(
\frac{ \delta }{ 1-\delta }
\right)^2
(1 + \delta),
\]
so that letting $c(\delta) = \sqrt{1+\delta}/(1-\delta)$
\[
\norm{ r_{M,i} }
\leq
\delta c(\delta),
\]
and
\begin{align*}
\norm{w_{M,i} - X_i} = \norm[2]{
\frac{v_{M,i}}{\norm{ v_{M,i}}}
-
X_i }
& =
\norm[2]{
\frac{v_{M,i}}{ \norm{v_{M,i}} }
\left(
1 - \norm{v_{M,i}}
\right)
+v_{M,i}
-X_i}\\
&  \leq
2 \norm{r_{M,i} },
\end{align*}
from two applications of the triangle inequality,
and using that $\norm{X_i}=1$ since we assumed that $G_M$ has ones on its diagonal for all $M$.
Hence, we have
  \begin{align}
  \e{\gamma_{\cM_s,\infty}}
  & = \e[3]{\sup_{M \in \cM_s; i \in M}
    | w_{M,i}^t \xi | } \notag \\
  & \leq \e[3]{
    \sup_{M \in \cM_s; i \in M}
    | X_i^t \xi | }
    + \e[3]{\sup_{M \in \cM_s; i \in M}
    \left|
\left(
    w_{M,i}
-
    X_i
    \right)^t
    \xi
    \right|} \notag \\
  & \leq
    \e[3]{\sup_{i =1,\ldots,p}
    | X_i^t \xi | } \notag \\
  & \quad + 2 \delta c(\delta)
    \e[3]{\sup_{M \in \cM_s; i \in M}
    \abs[3]{
    \left(
    \frac{
    w_{M,i}
    -
    X_i}
    {
    \norm{w_{M,i}
    -
    X_i}}
\right)^t
\xi }} \label{eq:for-B-alpha} \\
& \leq
\sqrt{2 \log (2p)} 
+ 2 \delta c(\delta)
\sqrt{ 2s \log(6p / s)},  \notag 
\end{align}
where in the last step we have used Lemma~\ref{lm:simplecard}.
\end{proof}

\begin{proof}[Proof of Lemma \ref{lm:rip-constant}]
\label{proof:rip-constant}
Since $\|X_i\| = 1$ for $i=1,...,p$ we have $\mathrm{corr}(X^tX) = X^tX$ and so $\kappa(X,s) = \delta(X,s)$.
The Gram matrix in \eqref{eq:gram-xck} can be written as $I_p + c U_{p,k}$, where $U_{p,k}$ is the $3 \times 3$ block matrix with sizes $(k, p-k-1, 1) \times (k, p-k-1, 1)$ defined by

\begin{equation*}
 U_{p,k} =
\begin{bmatrix}
    [0] &  [0] & [1] \\
    [0] & [0] & [0]\\
    [1] & [0] & 0\\
\end{bmatrix}.
\end{equation*}

Consider a model $M$ with $|M| =s \leq k \leq p-1$, and denote by $G_M$ its Gram matrix.  If $p \notin M$,  then $G_M=I_s$ and $\norm{G_M - I_s}_{op}=0$. If $p \in M$, then $G_M = I_s + c U_{s,\mk}$, where $m=\mk(M) = \left| (M \setminus \{p\}) \cap \{1, \dots k\} \right| \leq s-1$.  The operator norm of $G_M - I_s$ is the square root of the largest eigenvalue of $(cU_{s,\mk})^2$, where $U_{s,\mk}^2$ is a $3 \times 3$ block matrix with sizes $(\mk, s-\mk-1, 1) \times (\mk, s-\mk-1, 1)$ defined by

\begin{equation*}
 U_{s,\mk}^2 =
\begin{bmatrix}
     [1] &  [0] & [0] \\
    [0] & [0] & [0]\\
    [0] & [0] & \mk\\
\end{bmatrix}.
\end{equation*}

The first block is a $\mk \times \mk$ matrix with all entries equal to 1, hence its only non-null eigenvalue is $\mk$. This is also the (only) eigenvalue of the last block (an $1 \times 1$ matrix). Thus, the largest eigenvalue of $U_{s,\mk}^2$ is $\mk$. Therefore, as $\mk \leq s-1$, we have $\Vert G_{M} - I_s \Vert_{op} = c\sqrt{s-1}$ for all $M$ such that  $|M| =s \leq k \leq p-1$, which concludes the proof.
\end{proof}

\begin{proof}[Proof of Proposition~\ref{prop:matching-lower-bound}]
\label{proof:matching-lower-bound}
Without loss of generality (by Lemma~\ref{lem:posi:corrX}) we can assume that $X = Z^{(c,k)}$, where $Z^{(c,k)}$ is the $p \times p$ matrix defined as the beginning of Section~\ref{sec:equi-corr-design}.
The proof is an extension of the proof of  \cite[Theorem 6.2]{Berk13}.
For $m\geq 0$, consider a model $M$ such that $M \ni p$,
$M \cap \set{k+1,\ldots,p-1} = \emptyset$, and $\abs{M}=m+1$; in other words,
$M=\set{i_1,\ldots,i_m,p}$ such that $i_1,\ldots,i_m$ are elements of $\set{1,\ldots,k}$.
Denote as $\cM^{+p}_{m:k}$ the set of all such models.
Let $u_{M,p} = Z_{p} - P_{M\setminus\set{p}}( Z_p )$, where $Z_p$ is the last column of $Z^{(c,k)}$, and where $P_{M\setminus\set{p}}( Z_p )$ is the orthogonal projection of $Z_p$ onto the span of the columns with indices $M\setminus\set{p}$.
Observe that the column $i_j$ of $Z^{(c,k)}$ is the $i_j$-th base column vector of $\mathbb{R}^p$ that we write $e_{i_j}$, therefore 
\[
 P_{M\setminus\set{p}}(Z_p)  = \sum_{j=1}^m (e_{i_j}^t Z_p) e_{i_j} = c( e_{i_1} + \ldots + e_{i_m}).
 \]
 Hence, we have, for $M \in \cM^{+p}_{m:k}$,
\[
  \brac[1]{u_{M,p} }_j
=
\begin{cases}
0 & \text{ for } j=k+1,\ldots,p-1, \\
0 & \text{ for } j=1,\ldots,k; j \in M, \\
c &  \text{ for } j=1,\ldots,k; j \not \in M, \\
\sqrt{ 1 - k c^2 } & \text{ for } j=p.
\end{cases}
\]
Recall that we have $w_{M,p} = u_{M,p} / \|u_{M,p}\|$. Hence,
for $M \in \cM^{+p}_{m:k}$,
\[
\left[ w_{M,p} \right]_j
=
\begin{cases}
0 &\text{ for } j=k+1,\ldots,p-1, \\
0 & \text{ for } j=1,\ldots,k; j \in M, \\
{c/\sqrt{ 1 - mc^2 }} & \text{ for } j=1,\ldots,k; j \not \in M, \\
\sqrt{ 1 - k c^2 } /{\sqrt{ 1 - mc^2 }}
& \text{ for } j=p.
\end{cases}
\]
Hence, we have
\begin{align*}
\e{\gamma_{\cM_s,\infty}} & =
\e{\max_{|M| \leq s , i \in M}
| w_{M,i}^t \xi |} \\
& \geq
\e[3]{\max_{M \in \cM^{+p}_{(s-1):k}}
 w_{M,p}^t \xi } \\
& =\e[4]{
\frac{ \sqrt{ 1 - k c^2 } }{\sqrt{ 1 - (s-1)c^2 }}
\xi_p
+ \frac{c}{\sqrt{ 1 - (s-1)c^2 }}
\sum_{j=1}^{k-s+1} \xi_{k-j:k}
                      },
\end{align*}
where $\xi_{1:k} \leq \ldots \leq \xi_{k:k}$ are the order statistics of $\xi_1,\ldots,\xi_k$. Hence, since $s-1 < k$, we obtain
\begin{align*}
\e{\gamma_{\cM_s,\infty}} & \geq
0
+
\frac{c}{\sqrt{ 1 - (s-1)c^2 }}
\e[4]{\sum_{j=1}^{k} \xi_{j} -
\sum_{j=1}^{s-1} \xi_{j:k}}
  \\
& =
\frac{c}{\sqrt{ 1 - (s-1)c^2 }}
       \e[4]{\sum_{j=1}^{s-1} \xi_{k-j:k}}\\
       & \geq \frac{c}{\sqrt{ 1 - (s-1)c^2 }} \e[4]{\sum_{j=1}^{s-1}
\max_{l=1,\ldots,\lfloor k / s \rfloor}
\xi_{ (j-1) \lfloor k/s \rfloor + l }}.
\end{align*}
In the above display, each maximum has mean value larger than $A \sqrt{\log  \lfloor k/s \rfloor}$, with $A>0$ a universal constant (see e.g. Lemma A.3 in \cite{chatterjee2014superconcentration}).
Hence, we have
\[
  \e{\gamma_{\cM_s,\infty}}  \geq A
\frac{c(s-1)}{\sqrt{ 1 - (s-1)c^2 }}  \sqrt{ \log  \lfloor k/s \rfloor }.
\]
Finally, a consequence of Gaussian concentration  (Theorem~\ref{thm:gausconc})
is that mean and median of $\gamma_{\cM_s,\infty}$ are within $\sqrt{2\log 2}$ of
each other. Since we assumed $\alpha \leq \frac{1}{2}$,
$K(Z^{(c,k)},\cM_s,\alpha,\infty) \geq \e{\gamma_{\cM_s,\infty}} - \sqrt{2\log 2}$,
which concludes the proof.
\end{proof}

\begin{proof}[Proof of Corollary \ref{cor:optimality:bound}]

When $\delta_p \sqrt{s_p} \sqrt{ 1 - \log s_p / \log p + 1 / \log p } \to \infty$, one can see that in Theorem \ref{thm:rip-upper-bound}, the first term is negligible compared to the second one. Since $\delta_p \to 0$, the first result \eqref{eq:2} follows from Theorem \ref{thm:rip-upper-bound}.

We now apply Proposition~\ref{prop:matching-lower-bound} with $c_p = \delta_p / \sqrt{s_p  - 1}$ and $k_p = \min(p-1, \lfloor 1/c_p^2 - 1 \rfloor )$.
From Lemma \ref{lm:rip-constant}, $\delta(Z^{(c_p,k_p)}, s_p) \leq c_p \sqrt{s_p - 1} = \delta_p$.
We then have, with two positive constants $A'$ and $A$,
\begin{align*}
 K(Z^{(c_p,k_p)}, \cM_s,\alpha, \infty)
 \geq &
 A'
 \delta_p \sqrt{s_p}
 \sqrt{
\log \left(
\left\lfloor
\frac{\min( p-1 , \lfloor 1/c_p^2 - 1 \rfloor) }{ s_p}
\right\rfloor
\right)
}
\\
\geq &
A \delta_p \sqrt{s_p}
\sqrt{
\log
\left(
\min( \lfloor (p-1) /s_p \rfloor , 1/\delta_p^2 )
\right).
}
\end{align*}
This concludes the proof of \eqref{eq:1}.
\end{proof}

\section{Code for computing \texorpdfstring{$B_\ell(q, r, \rho)$}{Bl(q,r,rho)}}
\label{apx:code}

\begin{verbatim}
Bl <- function(q, r, rho, l, I = 1000) {
    ##
    ## Compute an upper bound for the quantile 1-l of
    ## max_{i=1,...,rho} (1/N) | w_i' V |
    ## where:
    ##   - the w_1,...w_{rho} are unit vectors
    ##   - V follows N(0,I_q)
    ##   - N^2/r follows X^2(r)
    ##
    ## Adapted from K4 in Bachoc, Leeb, Poetscher 2018
    ##
    ## Parameters:
    ## q.......: dimension of the Gaussian vector
    ## r.......: degrees of freedom for the variance estimator
    ## rho.....: number of unit vectors
    ## l.......: type I error rate (1 - confidence level)
    ## I.......: numerical precision
    ##
    ## Value:
    ##   A numerical approximation of the upper bound
    ##
    ##  vector of quantiles of Beta distribution:
    vC <-  qbeta(p = seq(from = 0, to = 1/rho, length = I),
                 shape1 = 1/2, shape2 = (q-1)/2,
                 lower.tail = FALSE)
    ## Monte-Carlo evaluation of confidence level 
    ## for a constant K
    fconfidence <- function(K){
        prob <- pf(q = K^2/vC/q, df1 = q, 
                   df2 = r, lower.tail = FALSE)
        mean(prob) - l
    }
    quant <- qf(p = l, df1 = q, df2 = r, lower.tail = FALSE)
    Kmax <- sqrt(quant) * sqrt(q)
    uniroot(fconfidence, interval = c(1, 2*Kmax))$root
}


\end{verbatim}

\addcontentsline{section}{toc}{References}
\bibliographystyle{abbrv}
\bibliography{Biblio}

\end{document}